\documentclass[12pt]{l4dc2020} 

\title[Reinforcement Learning of Markov Decision Processes with Peak Constraints]{Reinforcement Learning of Markov Decision Processes with Peak Constraints}
\usepackage{times}
\usepackage[T1]{fontenc}    % use 8-bit T1 fonts
\usepackage{hyperref}       % hypURL typesetting
\usepackage{booktabs}       % professional-quality tableserlinks
\usepackage{url}            % simple 
\usepackage{amsfonts}       % blackboard math symbols

\usepackage{graphicx} % more modern

\usepackage{hyperref}

\usepackage{amssymb}
\usepackage{verbatim}
\usepackage{graphicx}					% Hantera bilder
\usepackage{listings} 					% Skriva kod "snyggt"
\usepackage{amsmath}					% Matte
\usepackage{amssymb}					% Matte

\newcommand{\one}{\mathbf{1}}

\newcommand{\R}{\mathbb{R}}
\newcommand{\Prob}{\mathbf{Pr}}
\newcommand{\E}{\mathbf{E}}
\DeclareMathOperator*{\argmax}{arg\,max}

\newtheorem{defi}{Definition}

\newtheorem{thm}{Theorem}

\newtheorem{prp}{Proposition}

\newtheorem{assump}{Assumption}

\author{%
 \Name{Ather Gattami} \Email{ather.gattami@ri.se}\\
 \addr RISE AI\\ Research Institutes of Sweden\\ Stockholm, Sweden
}

\begin{document}

\maketitle

\begin{abstract} 
	In this paper, we consider reinforcement learning of Markov Decision Processes (MDP) with peak constraints, where an agent chooses a policy to optimize an objective and at the same time satisfy additional peak constraints. The agent has to take actions based on the observed states, reward outputs, and constraint-outputs, without any knowledge about the dynamics, reward functions, and/or the knowledge of the constraint functions. We introduce a transformation of the original problem in order to apply reinforcement learning algorithms where the agent maximizes a bounded and unconstrained objective. We show that the policies obtained from the transformed problem are optimal whenever the original problem is feasible. Out solution is memory efficient and doesn't require to store the values of the constraint functions. To the best of our knowledge, this is the first time learning algorithms guarantee convergence to optimal stationary policies for the MDP problem with peak constraints for discounted and expected average rewards, respectively.
\end{abstract} 


\section{Introduction}
\subsection{Motivation}
%Reinforcement learning has made great advances in several applications, ranging from online learning and recommender engines, natural language understanding and generation, to mastering games such as Go \cite{silver:2017} and Chess. The idea is to learn from extensive experience how to take actions that maximize a given reward by interacting with the surrounding environment. The interaction teaches the agent how to maximize its reward without knowing the underlying dynamics of the process. A classical example is swinging up a pendulum in an upright position. By making several attempts to swing up a pendulum and balancing it, one might be able to learn the necessary forces that need to be applied in order to balance the pendulum without knowing the physical model behind it, which is the general approach of classical model based control theory \cite{astrom:1994}.
Reinforcement learning is concerned with optimizing an objective function that depends on a given agent's action and the state of the process to be controlled. However, many applications in practice require that we take actions that are subject to additional constraints that need to be fulfilled. One example is wireless communication where the total transmission power of the connected wireless devices is to be minimized subject to constraints on the quality of service (QoS) such as maximum delay constraints \cite{djonin:2007}. Another example is the use of reinforcement learning methods to select treatments for future patients. To achieve just the right effect of a drug for specific patients, one needs to make sure that other important patient values satisfy some peak constraints that should not be violated in the short or long run for the safety of the patient.

Informally, the problem of reinforcement learning for Markov decision processes with peak constraints is described as follows (note that bandit optimization with peak constraints becomes a special case). Given a stochastic process with state $s_k$ at time step $k$, 
reward function $r$, and a discount factor $0<\gamma<1$, the constrained reinforcement learning problem is that for the optimizing agent to find a stationary policy $\pi(s_k)$, taking values in some finite set $A$, that minimizes the discounted reward 
\begin{equation}
\label{dis}
\sum_{k=0}^{\infty}\gamma^k \E\left(r(s_k, \pi(s_k))\right)
\end{equation}
or the expected average reward
\begin{equation}
\label{ave}
\lim_{T\rightarrow \infty} \frac{1}{T}\sum_{k=0}^{T-1} \E\left(r(s_k, \pi(s_k))\right) 
\end{equation}
subject to the constraints
\begin{equation}
\label{constr}
%\E(
r^j(s_k, a_k) \ge 0, ~~~ a_k \in A, ~~~\textup{for all } k, ~~~  j=1, ..., J
\end{equation}
(a more formal definition of the problem is introduced in the next section).

The \textit{peak} constrained reinforcement learning problem of Markov decision processes is that of finding an optimal policy that satisfies a number of peak constraints of the form (\ref{constr}). The constraints (\ref{constr}) can be equivalently written as $a_k\in A(s_k)$, where the feasibility set $A(s_k)$ is state dependent and is given by the inequality constraints (\ref{constr}). 
If the functions $r_1, .., r_j$ where known, then the optimization problem is straightforward and standard $Q$-learning solves the problem in the case of unknown process and reward $r$. However, if we don't know the constraint functions $r_1, ..., r_J$, then we don't know the feasibility set $A(s_k)$, and hence we need to learn it. It's important to note that since the constraint functions are unknown, it's inevitable that we violate the constraints in the learning process. The agent may, however, measure the samples $r_1(s_k,a_k), ..., r_J(s_k,a_k)$. 
The goal of this paper is to provide an algorithm that asymptotically converges to a feasible and optimal solution. One way could be to wait until one passes through all pairs $(s,a)$ to observe $r_1(s,a), ..., r_J(s,a)$, learn them, and then construct the feasibility sets $A(s)$ and run $Q$-learning for optimizing. However, this imposes huge cost in terms of memory, as one must store $r_1(s,a), ..., r_J(s,a)$ for all $(s,a)$, which mounts to $S\times A\times J$ elements. Furthermore, not optimizing along the learning process could be costly.

The following example from wireless communication describes in more detail a model where we have a Markov decision process with hard (peak) constraints and where the agent doesn't have knowledge of the process and reward functions, but only observations of the state $s$ and reward samples $r(s,a), r_1(s,a), ..., r_J(s,a)$. 
\begin{example}[Wireless communication]
	Consider the problem of wireless communication were the goal is to minimize the average of the transmitted
	power subject to a strict quality of service (QoS) constraint. Let $s_k$ denote the channel state at time step $k$ which belongs to a finite set and let the $a_k$ be the bandwidth allocation action, also belonging to a finite set of actions. The power required to occupy a bandwidth $a$ given the channel state is $P(s,a)$, which is unknown to the agent. The power affects the channel state, and hence, the channel evolves according to a probability distribution given by $p(s_{k+1}\mid s_k, a_k)$ which is also unknown. The QoS is given by a lower bound $b$ on the bit error rate, given by $q(s_k, a_k)\ge b$. The function $q(s_k, a_k)$ is not known as it is affected by noise that is not accessible to the agent. 
	By introducing $r(s,a) = -P(s,a)$ and $r^1(s,a) = q(s,a)-b$, the task is to solve the following optimization problem
	\[
	\begin{aligned}
	\sup_{a_k} ~& \lim_{N\rightarrow \infty} \frac{1}{N}\sum_{k=0}^{N-1} \E(r(s_k, a_k) ) \\
	\textup{s. t.}~ &~  r^1(s_k, a_k) \ge 0
	\end{aligned}
	\] 
\end{example}
\begin{example}[Search Engine]
	In a search engine, there is a number of documents that are related to a certain query. There are two values that are related to every document, the first being a (advertisement) value $u_i$ of document $i$ for the search engine and the second being a value $v_i$ for the user (could be a measure of how strongly related the document is to the user query). The task of the search engine is to display the documents in a row some order, where each row has an attention value, $A_j$ for row $j$. We assume that $u_i$ and $v_i$ are known to the search engine for all $i$, whereas the attention values $\{A_j\}$ are not known.
	%The strategy $\pi$ of the search engine is to display document $i$ in position $j$, $\pi(i) = j$, with probability $p_{ij}$. 
	The action of the search engine is to display document $i$ in position $j$, $a_i = j$. Thus, the expected average reward for the search engine is
	$$
	R = \lim_{N\rightarrow \infty} \frac{1}{N}\sum_{i=1}^N \E(u_i A_{a(i)}) 
	$$
	%and for the user
	%$$
	%R^1 = \lim_{N\rightarrow \infty} \frac{1}{N}\sum_{i=1}^N \E(v_i A_{\pi(i)}) 
	%$$
	The search engine has multiple objectives here where it wants to maximize the rewards for the user and itself. One solution is to define a measure for the quality of service, $v_i A_{a(i)}\ge q$ and then maximize its reward subject to the quality of service constraint, that is
	\[
	\begin{aligned}
	\sup_{a_i} ~& \lim_{N\rightarrow \infty} \frac{1}{N}\sum_{i=1}^N \E(u_i A_{a(i)}) \\
	\textup{s. t.}~ &  
	%\lim_{N\rightarrow \infty} \frac{1}{N}\sum_{i=1}^N 
	%\E(v_i A_{\pi(i)}) 
	v_i A_{a(i)} - q 
	\ge 0
	\end{aligned}
	\]
	
\end{example}
Although constrained Markov decision process problems are fundamental and have been studied extensively in the literature (see \cite{Altman99constrainedmarkov} and the references therein), the reinforcement learning counter part  of \textit{finding the optimal policies} seem to be still open, and even less is known for the case of peak constraints considered in this paper. When an agent has to take actions based on the observed states, rewards outputs, and constraint-outputs solely (without any knowledge about the dynamics, reward functions, and/or the knowledge of the constraint functions), a general solution seem to be lacking to the best of the authors' knowledge.
%In certain application, one might want to impose constraints on some reward functions for each time step. For the wireless communication example, there are certain applications in the 5G network architecture were high reliability/low latency requirements impose strict delay constraints at every time step \cite{johansson:2015}. That is, we will need more strict  constraints of the form
%\begin{equation*}
%\E(r^j(s_k, \pi(s_k))) \ge 0, ~~~\textup{for all } k, ~~~  j=1, ..., J-1
%\end{equation*}
%where the expectation is taken with respect to $\pi$ 
%(a more formal definition of the problem is introduced in the next section). 

\subsection{Previous Work}
Most of the work on Markov decision processes and bandit optimization considers constraints in the form of discounted or expected average rewards \cite{Altman99constrainedmarkov}. 
Constrained MDP problems are convex and hence one can convert the constrained MDP problem to an unconstrained zero-sum game where the objective is the Lagrangian of the optimization problem \cite{Altman99constrainedmarkov}.  However, when the dynamics and rewards are not known, it doesn't become apparent how to do it as the Lagrangian will itself become unknown to the optimizing agent. 
Previous work regarding constrained MDP:s, when the dynamics of the stochastic process are not known, considers scalarization through weighted sums of the rewards, see  \cite{Roijers:2013:SMS:2591248.2591251} and the references therein. 
Another approach is to consider Pareto optimality when multiple objectives are present \cite{JMLR:v15:vanmoffaert14a}.
However, none of the aforementioned approaches guarantee to satisfy lower bounds for a given set of reward functions simultaneously. In \cite{Gabor98multi-criteriareinforcement}, a multi-criteria problem is considered where the search is over deterministic policies. In general, however, deterministic policies are not optimal \cite{Altman99constrainedmarkov}. Also, the multi-criteria approach in 
\cite{Gabor98multi-criteriareinforcement} may provide a deterministic solution to a multi-objective problem in the case of two objectives and it's not clear how to generalize to a number of objectives larger than two.
In \cite{geibel:2006}, the author considers a single constraint and allowing for randomized policies. However,  no proofs of convergence are provided for the proposed sub-optimal algorithms. Sub-optimal solutions with convergence guarantees are provided in \cite{chow:2017} for the single constraint problem, allowing for randomized polices. In \cite{borkar:2005}, an actor-critic sub-optimal algorithm is provided for one single constraints and it's claimed that it can generalized to an arbitrary number of constraints. 
Sub-optimal solutions to constrained reinforcement learning problems with expected average rewards in a wireless communications context were considered in \cite{djonin:2007}. Sub-optimal reinforcement learning algorithms were presented in \cite{lizotte:2010} for controlled trial analysis with multiple rewards, again by considering a scalarization approach. In \cite{drugan:2013}, multi-objective bandit algorithms were studied by considering scalarization functions and Pareto partial orders, respectively, and present regret bounds. As previous results, the approach in \cite{drugan:2013} doesn't guarantee to satisfy the constraints that correspond to the multiple objectives. In \cite{achiam:2017}, constrained policy optimization is studied for the continuous MDP problem and some heurestic algorithms were suggested.

\subsection{Contributions}
We consider the problem of optimization and learning for Markov decision processes with peak constraints given by (\ref{constr}), for both discounted and expected average rewards, respectively. We reformulate the optimization problem with peak constraints to a zero-sum game where the opponent acts on a \textit{finite} set of actions (and not on a continuous space of actions). This transformation is essential in order to achieve a tractable optimal algorithm. The reason is that using Lagrange duality without model knowledge requires infinite dimensional optimization since the Lagrange multipliers are continuous (compare to the intractability of a partially observable MDP, where the beliefs are continuous variables). Furthermore, the Lagrange multipliers imply that the reward function is unbounded. We introduce a reformulation of the problem that is proved to be equivalent, where we get bounded reward functions and thereafter apply reinforcement learning algorithms that converge to optimal policies for the Markov decision process problem with peak constraints. The algorithm works without imposing a cost in terms of memory, where there is no need to store the learned values of $r_1(s,a), ..., r_J(s,a)$ for all $(s,a)$, saving memory of size $S\times A\times J$. 
We give complete proofs for both cases of discounted and expected average rewards, respectively.

\subsection{Notation}
\begin{tabular}{ll}
	$\mathbb{N}$ & The set of nonnegative integers.\\
	$[J]$ & The set of integers $\{1, ..., J\}$.\\
	$\mathbb{R}$ & The set of real numbers.\\
	$\mathbf{E}$ & The expectation operator.\\
	$\mathbf{Pr}$ &  $\mathbf{Pr}(x\mid y)$ denotes the probability of the\\
	&  stochastic variable $x$ given $y$.\\
	$\argmax$ 		& $\pi^\star = \argmax_{\pi\in \Pi} f_\pi$ denotes an 			
	element\\
	&  $\pi^\star \in 	\Pi$ that maximizes the function $f_\pi$.\\
	$\ge$ 				& For $\lambda = (\lambda^1, ..., \lambda^{J})$, 
	$\lambda \ge 0$ denotes\\ 
	& that $\lambda^i\ge 0$ for $i=1, ..., J$.	\\
	$1_{\{X\}}(x)$ & $1_{\{X\}}(x) = 1$ if $x\in X$ and\\ 
	& $1_{\{X\}}(x) = 0$ if $x\notin X$.\\
	$\mathbf{1}_n$ & $\mathbf{1}_{n} = (1, 1, ..., 1)\in \R^n$. \\
	$N(t, s, a)$ & $N(t, s, a) = \sum_{k=1}^t 1_{\{s, a\}}(s_k, a_k)$. \\
	e & $\textup{e}: (s,a) \mapsto 1$.\\
	$|S|$ & Denotes the number of elements in $S$.\\
	$s_+$        & For a state $s = s_k$, we have $s_+ = s_{k+1}$.\\
	%$[x]^+$ & $\max(0, x)$
\end{tabular}

\section{Problem Formulation}
\label{probform}
Consider a Markov Decision Process (MDP) defined by the tuple $(S, A, P)$, where  $S = \{S_1, S_2, ..., S_n\}$ is a finite set of states, $A = \{A_1, A_2, ..., A_m\}$ is a finite set of actions taken by the agent, and $P:S\times A \times S \rightarrow [0,1]$ is a transition function mapping each triple $(s, a, s_+)$ to a probability given by
$$P(s, a, s_+)=\mathbf{Pr}(s_+ \mid s, a)$$
and hence,
$$
\sum_{s_+\in S} P(s, a, s_+) = 1, ~~~ \forall (s,a)\in S\times A
$$
Let $\Pi$ be the set of policies that map 
a state $s\in S$ to a probability distribution of the actions with a probability assigned to each action $a\in A$, that is $\pi(s) = a$ with probability $\Prob(a\mid s)$.
The agent's objective is to find a stationary policy $\pi\in \Pi$ that maximizes the expected value of the total the discounted reward 
\begin{equation}
\label{discounted}
\sum_{k=0}^{\infty}\gamma^k \E\left(r(s_k, \pi(s_k))\right)
\end{equation}
or the expected average reward
\begin{equation}
\label{average}
\lim_{T\rightarrow \infty} \frac{1}{T}\sum_{k=0}^{T-1} \E\left(r(s_k, \pi(s_k))\right) 
\end{equation}
for $s_0 = s\in S$, where $r:S\times A \rightarrow \R$ is some unknown reward function. The parameter $\gamma \in (0,1)$ is a discount factor which models how much weight to put on future rewards. The expectation is taken with respect to the randomness introduced by the policy $\pi$ and the transition mapping $P$. 

The (hard or peak) constrained  reinforcement learning problem is concerned with finding a policy that satisfies a set of peak constraints given (\ref{constr}), 
where $r^j:S\times A \rightarrow \R$ are bounded functions, for $j=1, ..., J$. The agent doesn't have knowledge of the process and reward functions, but only measurements of the state $s$ and reward samples $r(s_k,a_k), r_1(s_k,a_k), ..., r_J(s_k,a_k)$. However, the agent knows that the reward functions are bounded by some constant $c$.

\begin{assump}
	\label{r}
	The absolute values of the reward functions $r$ and $\{r^j\}_{j=1}^{J}$ are bounded by some constant $c$ known to the agent. %Furthermore, $r(s,a) > 0$ for all $(s, a)$.
\end{assump}

\begin{assump}
	\label{rp}
	$r(s,a) > 0$ for all $(s, a)\in S\times A$.
\end{assump}
Note that Assumption \ref{rp} is not restrictive since we can replace $r$ with $r+c+\epsilon$ for some positive real number $\epsilon>0$ and obtain an equivalent problem with positive reward $r$ in the objective function and it will remain bounded by the constant $2c+\epsilon$.

%We will also assume that we know the largest negative value over all functions $r^1, ..., r^j$.
%\begin{definition}
%	\label{rj}
%	We define $\delta$ as the the absolute value of the maximum negative value over the functions $r^1, ..., r^j$,  that is
%	$\delta = \min \{|r^j(s,a)|\mid (s,a,j)\in (S,A,[J]), r^j(s,a)<0\}$. %is known to the agent.
%\end{definition}

\section{Reinforcement Learning for Markov Decision Processes}
\label{RL}

\subsection{Discounted Rewards}
Consider a Markov Decision Process where the agent is maximizing the total discounted reward given by
\begin{equation}
\label{minreward}
V(s_0) = \sum_{k=0}^{\infty}\gamma^k \E\left(R(s_k, a_k)\right) 
\end{equation} 
for some initial state $s_0\in S$. Let $Q^\star(s,a)$ be the expected reward of the agent taking action $a\in A$ from state $s\in S$, and continuing with an optimal policy thereafter.

%Let $\Pi_c$ be the set of policies that satisfy the peak constraints (\ref{constr}), that is $\pi(s_k) = a_k$ with probability $p(a_k \mid s_k)>0$ if and only if $r^j(s_k, a_k) \ge 0$ for all $j\in[J]$. 
Then for any stationary policy $\pi$, we have that
\begin{equation}
\label{Q} 
\begin{aligned}
Q^\star(s,a) 	&= R(s, a) + \max_{\pi\in \Pi}
\sum_{k=1}^{\infty}\gamma^k \E\left(R(s_k, \pi(s_k))\right)\\
%&=  r(s, a^\star)  + \gamma \cdot %\max_{\pi\in \Pi}
%\min_{o\in O}\E\left( \sum_{k=0}^{\infty}\gamma^{k} R(s_{k+1}, \pi(s_{k+1}))\right)\\
& = R(s, a) + \gamma \cdot  \max_{\pi\in \Pi}
\E\left(Q^\star(s_+, \pi(s_+))\right)\\
%&  = R(s, a)  + \gamma \cdot \max_{\pi\in \Pi}
%\min_{o_+\in O}\E\left(\sum_{s_+\in S} P(s,a,o,s_+)Q(s_+, \pi(s_+))\right)
\end{aligned}
\end{equation} 
%and
%\begin{equation}
%\label{Q} 
%\begin{aligned}
%\min_{o\in O} Q(s,a)   & = \min_{o\in O} \left(R(s, a, o) + \gamma \cdot 
% \max_{\pi\in \Pi} \min_{o\in O} \E\left(Q(s_+, \pi(s_+))\right) \right)\\
%%&  = R(s, a)  + \gamma \cdot \max_{\pi\in \Pi}
%%\min_{o_+\in O}\E\left(\sum_{s_+\in S} P(s,a,o,s_+)Q(s_+, \pi(s_+)_+)\right)
%\end{aligned}
%\end{equation} 
%{\small
%	\begin{equation}
%	\label{Q} 
%	\begin{aligned}
%	Q(s,a,o^\star) 	&= R(s, a^\star) + \max_{\pi\in \Pi}
%	\E\left( \sum_{k=1}^{\infty}\gamma^k R(s_k, \pi(s_k)^\star)\right)\\
%	%&=  R(s, a^\star)  + \gamma \cdot \max_{\pi\in \Pi}
%	%\min_{o\in O}\E\left( \sum_{k=0}^{\infty}\gamma^{k} R(s_{k+1}, \pi(s_{k+1}))\right)\\
%	& = R(s, a ^\star) + \gamma \cdot \max_{\pi\in \Pi}
%	\E\left(Q(s_+, \pi(s_+) ^\star)\right)\\
%	%&  = R(s, a )  + \gamma \cdot \max_{\pi\in \Pi}
%	%\min_{o_+\in O}\E\left(\sum_{s_+\in S} P(s,a,o,s_+)Q(s_+, \pi(s_+) _+)\right)
%	\end{aligned}
%	\end{equation} 
%}
%where
%$$
%o^\star = \arg \min_{o\in O}\E\left(Q(s, \pi(s) )\right)
%$$
Equation (\ref{Q}) is known as the Bellman equation,  
and the solution to (\ref{Q}) with respect to $Q^\star$ that corresponds to the optimal policy $\pi^\star$ and optimal actions of the opponent is denoted $Q^\star$.
If we have the function $Q^\star$, then we can obtain the optimal policy $\pi^\star$ according to the equation
\begin{equation}
\label{pistar0}
\begin{aligned}
\pi^\star(s) &= \argmax_{\pi\in \Pi} \E\left(Q^\star(s, \pi(s))\right) \\
%&= \argmax_{\pi\in \Pi} \min_{o\in O} \left(R(s, a ) + \gamma \cdot 
%\E\left(Q^\star(s_+, \pi(s_+) )\right) \right)
\end{aligned}
\end{equation}
which maximizes the total discounted reward 
\begin{equation*}
\begin{aligned}
\sum_{k=0}^{\infty}\gamma^k \E\left(R(s_k, \pi^\star(s_k))\right) = 
\E\left(Q^\star(s, \pi^\star(s))\right) 
\end{aligned}
\end{equation*} 
for $s=s_0$. Note that the optimal policy may not be deterministic, as opposed to reinforcement learning for unconstrained Markov Decision Processes, where there is always an optimal policy that is deterministic. 

In the case we don't know the process $P$ and the reward function $R$, we will not be able to take advantage of the Bellman equation directly. The following results show that we will be able to design an algorithm that always converges to $Q^\star$. 

\begin{defi}[Unichain MDP]
	An MDP is called unichain, if for each policy $\pi$ the Markov chain
	induced by $\pi$ is ergodic, i.e. each state is reachable from any other state.
\end{defi}
Unichain MDP:s are usually considered in reinforcement learning problems with discounted rewards,  since they guarantee that we learn the process dynamics from the initial states. Thus, for the discounted reward case we will make the following assumption. 
\begin{assump}[Unichain MDP]
	\label{unichain}
	The MDP $(S, A, P)$ is assumed to be unichain.
\end{assump}

\begin{prp}
	\label{algo}
	Consider a Markov Decision Process given by the tuple $(S, A, P)$, where $(S, A, P)$ is unichain, and suppose that $R$ is bounded. 
	Let $Q=Q^\star$ and $\pi = \pi^\star$ be solutions to the Bellman equation
	%	{\small
	%		\begin{equation}
	%		\label{bellman}
	%		\begin{aligned}
	%		Q(s,a,o^\star) 	&= R(s, a, o^\star) + \gamma \cdot \max_{\pi\in \Pi}
	%		\E\left(Q(s_+, \pi(s_+))\right)\\
	%		o^\star &= \arg \min_{o\in O}\E\left(Q(s, \pi(s) )\right)
	%		\end{aligned}
	%		\end{equation}		
	%	}
	\begin{equation}
	\label{bellman}
	\begin{aligned}
	Q(s,a) 	&= R(s, a) + \gamma \cdot \max_{\pi\in \Pi}
	\E\left(Q(s_+, \pi(s))\right)\\
	\pi(s) &= \argmax_{\pi\in \Pi}\E\left(Q(s, \pi(s))\right)
	\end{aligned}
	\end{equation}		
	Let $\alpha_k(s, a) = \alpha_k \cdot \one_{\{s,a\}}(s_k,a_k)$ satisfy
	\begin{equation}
	\label{alpha}
	\begin{aligned}
	&0\le \alpha_k(s, a) < 1\\
	&\sum_{k=0}^\infty \alpha_k(s,a) = \infty\\
	& \sum_{k=0}^\infty \alpha_k^2(s,a) < \infty\\
	& \forall (s,a) \in S\times A
	\end{aligned}
	\end{equation}
	Then, the update rule
	\begin{equation}
	\label{q-learning}
	\begin{aligned}
	%a_+	&= \argmax_{a_+\in A} \min_{o\in O} (R(s,a) +\gamma  Q_{k}(s_+,a_+) )\\
	Q_{k+1} (s,a) &= (1-\alpha_k(s,a))Q_k(s,a)+\alpha_k(s,a)(R(s,a)+\gamma  \max_{a\in A}  Q_{k}(s_+,a))
	\end{aligned}
	\end{equation}
	converges to $Q^\star$ with probability 1. 	Furthermore, the optimal policy $\pi^\star \in \Pi$ given by
	(\ref{pistar0}) maximizes
	(\ref{minreward}).
\end{prp}
\begin{proof}
	Consult \cite{jaakkola:1994}.
\end{proof}

\subsection{Expected Average Rewards}       
The agent's objective here is to maximize the total reward given by
\begin{equation}
\label{reward2}
\lim_{T\rightarrow \infty} \frac{1}{T}\sum_{k=0}^{T-1}  \E\left(R(s_k, a_k)\right) 
\end{equation} 
for some initial state $s_0\in S$.

We will make a simple assumption regarding the existence of recurring state, a standard assumption in Markov decision process problems with expected average rewards to ensure that the expected value of the reward is independent of the initial state.

\begin{assump}
	\label{recurrentstate}
	There exists a state $s^* \in S$ which is recurrent for every stationary policy $\pi$ played by the agent.
\end{assump}
Assumption \ref{recurrentstate} implies that $\E(r^j(s_k, a_k))$ is independent of the initial state at stationarity. 

\begin{prp}
	\label{stateindependent}
	If Assumption \ref{recurrentstate} holds, then the value of the Markov Decision Process $(S, A, P)$, with finite state and action spaces, is independent of the initial state. 
\end{prp}
\begin{proof}
	Consult \cite{bertsekas:2005}.
\end{proof}

\begin{prp}
	Under Assumption \ref{recurrentstate}, there exists a number $v$ and a vector $(h(S_1), ..., h(S_n))\in \mathbb{R}^{n}$, such that for each $s\in S$, we have that
	\begin{equation}
	\label{H}
	\begin{aligned}
	h(s)  + v &= \mathbf{T}h(s)
	= \max_{\pi\in\Pi} \E\Big(R(s, \pi(s)) + 
	\sum_{s_+\in S} P(s_+\mid s, \pi(s))h(s_+)\Big)
	\end{aligned}
	\end{equation} 
	Furthermore, the value of (\ref{reward2}) is equal to $v$.
\end{prp}
\begin{proof}
	Consult \cite{bertsekas:2005}.
\end{proof}
Similar to $Q$-learning (but still different), our goal is to find a function
$Q^\star(s,a)$ that satisfies
\begin{equation}
\label{Qreward2}
\begin{aligned}
Q^\star(s,a) + v^\star = & R(s, a) +  \sum_{s_+\in S} P(s_+\mid s, a)h^\star(s_+)
\end{aligned}
\end{equation} 
for any solutions $h^\star$ and $v^\star$ to Equations (\ref{H})-(\ref{Qreward2}). Note that
we have 
$$
h^\star(s) = \max_{\pi\in \Pi} \E(Q^\star(s,\pi(s)))
$$
In the case we don't know the process $P$ and the reward function $R$, we will not be able to take advantage of (\ref{Qreward2}) directly. The next proposition shows that we will be able to design an algorithm that always converges to $Q^\star$. 
It's worth to note here that the operator $\mathbf{T}$ in Equation (\ref{H}) is not a contraction, so the standard $Q$-learning that is commonly used for reinforcement learning in Markov decision processes with discounted rewards can't be applied here.

\begin{assump}[Learning rate]
	\label{lr}
	The sequence $\beta(k)$ satisfies:
	\begin{enumerate}
		\item For every $0<x<1$, $\sup_k \beta(\lfloor xk \rfloor)/ \beta(k)<\infty$
		\item $\sum_{k=1}^{\infty} \beta(k) = \infty$ and $\sum_{k=0}^{\infty} \beta(k)^2 < \infty$.
		\item For every $0<x<1$, the fraction
		$$
		\frac{\sum_{k=1}^{\lfloor yt\rfloor} \beta(k)}{\sum_{k=1}^t \beta(k)} %\rightarrow 1 
		$$
		converges to 1 uniformly in $y\in [x, 1]$ as $t\rightarrow \infty$.
	\end{enumerate}
\end{assump}
For example, $\beta(k) = \frac{1}{k}$ and $\beta(k) = \frac{1}{k\log k }$ (for $k>1$) satisfy Assumption \ref{lr}. 

Now define $N(t, s, a)$ as the number of times that state $s$ and actions $a$ and $b$ were played up to time $t$, that is 
$$
N(t, s, a) = \sum_{k=1}^t 1_{\{s, a\}}(s_k, a_k)
$$

The following assumption is needed to guarantee that all combinations of the triple $(s, a )$ are visited often.
\begin{assump}[Often updates]
	\label{ou}
	There exists a deterministic number $d>0$ such that for every $s\in S$, and $a\in A$, we have that
	$$
	\liminf_{t\rightarrow \infty} \frac{N(t, s, a)}{t} \ge d
	$$
	with probability 1.
\end{assump}

\begin{defi}
	\label{phi}
	We define the set $\Phi$ as the set of all functions $f:\R^{n\times m}\rightarrow \R$ such that
	\begin{enumerate}
		\item $f$ is Lipschitz
		\item For any $c\in \R$, $f(cQ) = cf(Q)$
		\item For any $r\in \R$ and $\widehat{Q}(s, a) = Q(s, a)+r$ for all 
		$(s, a)\in \R^{n\times m}$, we have $f(\widehat{Q}) = f(Q) + r$
	\end{enumerate}
\end{defi}

\begin{prp}
	\label{algo2}
	Suppose that $R$ is bounded and that Assumption \ref{recurrentstate}, \ref{lr}, and \ref{ou} hold. Introduce
	\begin{equation}
	\label{F}
	\mathbf{F}Q(s, a ) = \max_{\pi\in\Pi} \E \left(Q(s, \pi(s)) \right)
	\end{equation} 
	and let $f\in \Phi$ be given, where the set $\Phi$ is defined as in Definition \ref{phi}.
	Then, the asynchronous update algorithm given by 
	{
		\begin{equation}
		\label{optQ}
		\begin{aligned}
		&Q_{t+1}(s,a) = Q_{t}(s,a) + 1_{(s, a)}(s_t, a_t)\times  \\
		&~~~\times \beta (N(t, s, a))(R(s_t, a_t)
		+\mathbf{F} Q_t(s_{t+1}, a_t) - f(Q_t(s_{t}, a_t))- Q_t(s, a))
		\end{aligned}
		\end{equation}
	}
	converges to $Q^\star$ in  (\ref{H})-(\ref{Qreward2}) with probability 1.
\end{prp}
\begin{proof}
	Consult \cite{abounadi:2001}.
\end{proof}

\section{Reinforcement Learning for Constrained Markov Decision Processes}
\label{game-approach}
\subsection{Discounted Rewards}
Consider the optimization problem of finding a stationary policy $\pi$ subject to the initial state $s_0=s$ and the constraints (\ref{dis}), 
%\begin{equation}
%\label{constraint}
%\begin{aligned}
%\E(r^j(s_k, a_k)) \ge 0, ~~~~ \textup{for all } k\in \mathbb{N}, ~\textup{and } j=1, ..., J-1
%\end{aligned}
%\end{equation}
that is
\begin{equation}
\label{optproblem}
\begin{aligned}
\sup_{a_k\in A} ~~& \sum_{k=0}^{\infty}\gamma^k \E\left(r(s_k, a_k)\right)\\
\textup{s. t.}~~ 		&  %\E\left(
r^j(s_k, a_k)
%\right) 
\ge 0, ~~k\in \mathbb{N}\\ 
&  j=1, ..., J \\ 
& s_0 = s
\end{aligned}
\end{equation}
Since $\gamma>0$, optimization problem (\ref{optproblem}) is equivalent to
\begin{equation}
\label{optgamma}
\begin{aligned}
\sup_{a_k\in A} ~~& \sum_{k=0}^{\infty}\gamma^k \E\left(r(s_k, a_k)\right)\\
\textup{s. t.}~~ 		&  %\E\left(
\gamma^k r^j(s_k, a_k)
%\right) 
\ge 0, ~~k\in \mathbb{N}\\ 
&  j=1, ..., J \\ & s_0 = s
\end{aligned}
\end{equation}
%Note that the constraints (\ref{constraint}) are more strict than (\ref{discounted_constr}). %Hence, the value of (\ref{optproblem}) is more 
The Lagrange dual of the problem gives the following equivalent formulation
\begin{equation}
\label{dual}
\begin{aligned}
\sup_{a_k\in A} \min_{\lambda_k\ge 0}~~& \sum_{k=0}^{\infty}\gamma^k \left(\E(r(s_k, a_k))+\sum_{j=1}^J \lambda_k^j r^j(s_k,a_k)\right)\\
\textup{s. t.}~~ 		
%&  %\E\left(
%r^j(s_k, a_k)
%\right) 
%\ge 0\\ 
%&  j=1, ..., J \\ 
& s_0 = s
\end{aligned}
\end{equation}
Now introduce 
\begin{equation}
\label{Rlambda}
R(s,a, \lambda) = r(s, a)+\sum_{j=1}^J \lambda^j r^j(s,a)
\end{equation}
%and
%\begin{equation}
%\label{R}
%R(s,a) = \min_{\lambda\ge 0} \max\left(-C,~  r(s, a)+\sum_{j=1}^J \lambda_k^j r^j(s,a)\right)
%\end{equation}
%Consider a Markov Decision Process where the agent is maximizing the total discounted reward given by
%\begin{equation}
%\label{minreward}
%V(s) = \sum_{k=0}^{\infty}\gamma^k \E\left( \min_{\lambda_k\ge 0}R(s_k, a_k, \lambda_k)\right) 
%\end{equation} 
%for the initial state $s_0=s\in S$.
%and actions satisfying $r^j(s_k, a_k) \ge 0$ for $j\in[J]$
Let $Q^\star(s,a)$ be the expected reward of the agent taking action $a_0=a\in A$ from state $s_0=s$ satisfying the constraints in (\ref{optproblem}), and continuing with an optimal policy satisfying the constraints in (\ref{optproblem}) thereafter.
%Let $\Pi_c$ be the set of policies that satisfy the peak constraints (\ref{constr}), that is $\pi(s_k) = a_k$ with probability $p(a_k \mid s_k)>0$ if and only if $r^j(s_k, a_k) \ge 0$ for all $j\in[J]$. 
Then, %for any stationary policy $\pi$, 
we have that
\begin{equation}
\label{Q} 
\begin{aligned}
Q^\star(s,a) 	%&= \min_{\lambda\ge 0} R(s, a, \lambda) + \max_{\pi\in \Pi} \min_{\lambda_k\ge 0} \sum_{k=1}^{\infty}\gamma^k \E\left(R(s_k, \pi(s_k), \lambda_k)\right)\\
%&=  r(s, a^\star)  + \gamma \cdot \max_{\pi\in \Pi}
%\min_{o\in O}\E\left( \sum_{k=0}^{\infty}\gamma^{k} R(s_{k+1}, \pi(s_{k+1}))\right)\\
& = \min_{\lambda\ge 0} R(s, a, \lambda) + \gamma \cdot  \max_{\pi\in \Pi} 
\E\left(Q^\star(s_+, \pi(s_+))\right)\\
%&  = R(s, a)  + \gamma \cdot \max_{\pi\in \Pi}
%\min_{o_+\in O}\E\left(\sum_{s_+\in S} P(s,a,o,s_+)Q(s_+, \pi(s_+))\right)
\end{aligned}
\end{equation} 
and the optimal policy $\pi^\star$ is given by
$$
\pi^\star(s) = \argmax_{\pi\in \Pi}\E\left(Q^\star(s, \pi(s))\right)
$$

The next theorem states that if the optimization problem  (\ref{optproblem}) is feasible, then it's
equivalent to an unconstrained optimization problem with bounded rewards, in the sense that an optimal strategy of the agent in unconstrained problem for (\ref{optproblem}), and the value of the unconstrained problem is equal to the value of (\ref{optproblem}).

\begin{thm}
	\label{maxopt}
	Suppose that Assumption \ref{r} and \ref{rp} hold and set $C = c\cdot \gamma/(1-\gamma)$. 
	Let $Q^\star$ and $\pi^\star$ be solutions to
	\begin{equation}
	\label{bellman}
	\begin{aligned}
	Q^\star(s,a) 	&= \max\left(-C,~  \min_{\lambda\ge 0}R(s, a, \lambda)\right) + \gamma \cdot \max_{\pi\in \Pi}
	\E\left(Q^\star(s_+, \pi(s_+))\right)\\
	\pi^\star(s) &= \argmax_{\pi\in \Pi}\E\left(Q^\star(s, \pi(s))\right)
	\end{aligned}
	\end{equation}	
	If $Q^\star(s,a)>0$ for all $(s,a)\in S\times A$, then (\ref{dual}) is feasible and
	$\pi^\star$ is an optimal solution.% to (\ref{dual}). 
	Otherwise, if $Q^\star(s,a)\le 0$ for some 	
	$(s,a)\in S\times A$, then (\ref{dual}) is not feasible.  
\end{thm}
\begin{proof}
	See the appendix.
\end{proof}

The reason why we need Theorem \ref{maxopt} is that we want to transform the optimization problem (\ref{optproblem}) to an equivalent problem where the reward functions are bounded.

Now that we are equipped with Theorem \ref{maxopt}, we are ready to state and prove our first main result.

%\begin{theorem}
%	\label{mainthm0}	
%\end{theorem}

\begin{thm}
	\label{mainthm}
	Consider the constrained MDP problem (\ref{optproblem}) and suppose that it's feasible and that Assumption \ref{unichain}, \ref{r}, and \ref{rp} hold. Introduce $C=c\cdot \gamma/(1-\gamma)$ and
	\[
	R(s,a)  =\max\left(-C,~  \min_{\lambda\ge 0}R(s, a, \lambda)\right)
	\]
	Let $Q_k$ be given by the recursion according to  (\ref{q-learning}).  Then, $Q_k \rightarrow Q^\star$ as $k\rightarrow \infty$ where $Q^\star$ is the solution to (\ref{bellman}).
	Furthermore, the policy
	\begin{equation*}
	\label{pistar}
	\pi^\star(s) = \argmax_{\pi\in \Pi} \E\left(Q^\star(s, \pi(s) )\right)  
	\end{equation*}
	is an optimal solution to (\ref{optproblem}). 
\end{thm}
\begin{proof}
	Follows from Theorem \ref{maxopt}, that $R(s,a)$ is bounded for all $(s,a)\in S\times A$, and Proposition \ref{algo}.
\end{proof}

\subsection{Expected Average Rewards}
Consider the optimization problem of finding a stationary policy $\pi$ subject to the constraints (\ref{ave}), that is
%As we concluded  in the problem formulation section, Assumption \ref{recurrentstate} implies that the value of left hand side of the constraint (\ref{average_constr}) is independent of the initial state, and thus equal to $\E(r^j(s_k, \pi(s_k)))$, for all $j$.
%Thus, the constrained MDP problem with expected average rewards is, under Assumption \ref{recurrentstate}, equivalent to
\begin{equation}
\label{optproblem2}
\begin{aligned}
\sup_{a_k\in A} ~~& \lim_{T\rightarrow \infty}\frac{1}{T}\sum_{k=0}^{T-1}  \E\left(r(s_k, a_k)\right) \\
\textup{s. t.}~~ &  %\E\left(
r^j(s_k, a_k)
%\right)  
\ge 0\\ 
&\textup{for }  j=1, ..., J
\end{aligned}
\end{equation}
Under Assumption \ref{recurrentstate}, the value of the objective function is independent of the initial state according to Proposition \ref{stateindependent}.
%$$
%\lim_{T\rightarrow \infty}\frac{1}{T}\sum_{k=0}^{T-1}  \E\left(r(s_k, a_k)\right) = \E\left(r(s_k, a_k)\right)
%$$
%Thus, under Assumption \ref{recurrentstate2}, (\ref{optproblem2}) is equivalent to
%\begin{equation}
%\label{optE}
%\begin{aligned}
%\sup_{a_k\in A} ~~& \E\left(r(s_k, a_k)\right) \\
%\textup{s. t.}~~ &  %\E\left(
%r^j(s_k, a_k)
%%\right)  
%\ge 0\\ 
%&\textup{for }  j=1, ..., J
%\end{aligned}
%\end{equation}
The Lagrange dual of the problem gives the following equivalent formulation
\begin{equation}
\label{dual2}
\begin{aligned}
\sup_{a_k\in A} \min_{\lambda_k\ge 0}~~& \lim_{T\rightarrow \infty}\frac{1}{T}\sum_{k=0}^{T-1} \left(\E(r(s_k, a_k))+\sum_{j=1}^J \lambda_k^j r^j(s_k,a_k)\right)\\
%\textup{s. t.}~~ 		
%&  %\E\left(
%r^j(s_k, a_k)
%\right) 
%\ge 0\\ 
%&  j=1, ..., J \\ 
\end{aligned}
\end{equation}
Let $Q^\star(s,a)$ be the expected reward of the agent taking action $a_0=a\in A$ from state $s_0=s$ satisfying the constraints in (\ref{optproblem}), and continuing with an optimal policy satisfying the constraints in (\ref{optproblem}) thereafter.
Also, let $R(s,a, \lambda)$ be given by (\ref{Rlambda}).
Then, we have that
\begin{equation}
\label{Q2}
\begin{aligned}
h^\star(s)  + v^\star &= \max_{\pi\in\Pi} \Big(\min_{\lambda\ge 0} R(s, a, \lambda) + 
\sum_{s_+\in S} P(s_+\mid s, \pi(s))h^\star(s_+)\Big)\\
Q^\star(s,a) + v^\star &= \min_{\lambda\ge 0} R(s, a, \lambda) +  \sum_{s_+\in S} P(s_+\mid s, a)h^\star(s_+)
\end{aligned}
\end{equation} 
and the optimal policy $\pi^\star$ is given by
$$
\pi^\star(s) = \argmax_{\pi\in \Pi}\E\left(Q^\star(s, \pi(s))\right)
$$
The value of (\ref{optproblem2}) and (\ref{dual2}) is given by $v^\star$.

\begin{thm}
	\label{maxopt2}
	Suppose that Assumption \ref{r} and \ref{rp}. %hold and let $C = 2c$. 
	Also,  let $Q^\star$, $h^\star$, and $\pi^\star$ be solutions to
	\begin{equation}
	\label{bellmancut}
	\begin{aligned}
	h^\star(s)  + v^\star &= \max_{\pi\in\Pi} \E\Big(\max\left(-c,~  \min_{\lambda\ge 0}R(s, a, \lambda)\right) + 
	\sum_{s_+\in S} P(s_+\mid s, \pi(s))h^\star(s_+)\Big)\\
	Q^\star(s,a) + v^\star &= \max\left(-c,~  \min_{\lambda\ge 0}R(s, a, \lambda)\right) +  \sum_{s_+\in S} P(s_+\mid s, a)h^\star(s_+)\\
	\pi^\star(s) &= \argmax_{\pi\in \Pi}\E\left(Q^\star(s, \pi(s))\right)
	\end{aligned}
	\end{equation} 
	If $Q^\star(s,a)+ v^\star >0$ for all $(s,a)\in S\times A$, then (\ref{dual2}) is feasible and
	$\pi^\star$ is an optimal solution.% to (\ref{dual}). 
	Otherwise, if $Q^\star(s,a)+ v^\star \le 0$ for some 	
	$(s,a)\in S\times A$, then (\ref{dual2}) is not feasible.  
\end{thm}
\begin{proof}
	See the appendix.
\end{proof}

The reason why we need Theorem \ref{maxopt2} is that we want to transform the optimization problem (\ref{dual2}) to an equivalent problem where the reward functions are bounded.

Now that we are equipped with Theorem \ref{maxopt2}, we are ready to state and proof the second main result.

\begin{thm}
	\label{mainthm2}
	Consider the constrained Markov Decision Process problem   (\ref{optproblem2}) and suppose that Assumption \ref{recurrentstate}, \ref{lr}, and \ref{ou} hold. Introduce 
	\[
	R(s,a)  = \max\left(-c,~  \min_{\lambda\ge 0}R(s, a, \lambda)\right)
	\] 
	Let $Q_k$ be given by the recursion according to  (\ref{optQ}) and suppose that Assumptions \ref{lr} and \ref{ou} hold. Then, $Q_k \rightarrow Q^\star$ as $k\rightarrow \infty$ where $Q^\star$ is the solution to (\ref{bellmancut}).
	Furthermore, the policy
	\begin{equation}
	\label{pistar2}
	\pi^\star(s) = \argmax_{\pi\in \Pi} \E\left(Q^\star(s, \pi(s) )\right)  
	\end{equation}
	is a solution to (\ref{optproblem2}). %for all $s\in S$.
\end{thm}
\begin{proof}
	Follows from Theorem \ref{maxopt2}, that $R(s,a)$ is bounded for all $(s,a)\in S\times A$, and Proposition \ref{algo2}.
\end{proof}

\section{Conclusions}
\label{conc}
We considered the problem of optimization and learning for peak constrained Markov decision processes, for both discounted and expected average rewards, respectively. We transformed the original problems in order to apply reinforcement learning to a bounded unconstrained problem. The algorithm works without imposing a cost in terms of memory, where there is no need to store the learned values of $r_1(s,a), ..., r_J(s,a)$ for all $(s,a)$, saving memory of size $S\times A\times J$. 

It would be interesting to combine our algorithms with deep reinforcement learning and study the performance of the approach introduced in this paper when the value function $Q$ is modeled as a deep neural network.

\bibliography{../ref/mybib}
\newpage
\section*{Appendix}

\subsection*{Proof of Theorem \ref{maxopt}}
First note that the value of optimization problem (\ref{optproblem}) is positive if it's feasible since $r$ is positive according to Assumption \ref{rp}. Thus, the value of (\ref{dual}) is either positive if the constraints in (\ref{optproblem}) are satisfied,  or it goes to $-\infty$ if (\ref{optproblem}) is not feasible. Furthermore, the value of (\ref{dual}) (which is equal to the value of  (\ref{optproblem}) is at most $c/(1-\gamma)$ since
$r$ is bouned by $c$ according to Assumption \ref{r}. Thus, 
$$\gamma \max_{\pi\in \Pi} \E(Q^\star(s,\pi(s)))\le \gamma\cdot c/(1-\gamma) = C$$
for all $s\in S$. 
Suppose that $a$ is such that $r^j(s,a)<0$ for some $j$. Then, 
$$
\max\left(-C,~  \min_{\lambda\ge 0}R(s, a, \lambda)\right) = -C 
$$
and
$$
Q^\star(s,a) 	= \max\left(-C,~  \min_{\lambda\ge 0}R(s, a, \lambda)\right) + \gamma \cdot \max_{\pi\in \Pi}
\E\left(Q^\star(s_+, \pi(s_+))\right) \le 0
$$
On the other hand, if $a$ is such that $r^j(s,a)\ge 0$ for all $j\in [J]$, then we have that
$$
\max\left(-C,~  \min_{\lambda\ge 0}R(s, a, \lambda)\right) = \min_{\lambda\ge 0}R(s, a, \lambda)
$$
and
$$
Q^\star(s,a) 	= \min_{\lambda\ge 0}R(s, a, \lambda) + \gamma \cdot \max_{\pi\in \Pi} \E\left(Q^\star(s_+, \pi(s_+))\right)
$$
which is identical to (\ref{Q}). Thus, by choosing a policy that satisfies the constraints in (\ref{optproblem}), $r^j(s,a)\ge 0$, $Q^\star(s,a)$ will be positive. We conclude that if the policy $\pi^\star$
implies that $Q^\star(s,a)>0$, then it's optimal for (\ref{dual}). Otherwise at least one of the constraints $r^j(s,a)\ge 0$ is not satisfied and the value of 
$$\max_{\pi\in \Pi}\E(Q^\star(s,\pi(s)))=\E(Q^\star(s,\pi^\star(s)))$$ is non positive which implies that the value of (\ref{dual}) goes to $-\infty$, and so (\ref{optproblem}) is not feasible.

\subsection*{Proof of Theorem \ref{maxopt2}}
First note that the value of optimization problem (\ref{optproblem2}) is positive if it's feasible since $r$ is positive according to Assumption \ref{rp}. Thus, the value of (\ref{dual2}) is either positive if the constraints in (\ref{optproblem2}) are satisfied, or it goes to $-\infty$ if (\ref{optproblem2}) is not feasible. %Furthermore, the value of $Q^\star(s,a) + v^\star$ is positive  
%(\ref{dual2}) (which is equal to the value  $v$ of  (\ref{optproblem2}) is at most $c$ since
Now we have that $\min_{\lambda \ge 0} R(s, a, \lambda)\le c$ according to Assumption \ref{r}. Thus, 
%$c\ge v^\star \ge -c$ and
$$h^\star(s) = \max_{\pi\in \Pi} \E(Q^\star(s,\pi(s))) \le c$$
for all $s\in S$. 
Suppose that $a$ is such that $r^j(s,a)<0$ for some $j$. Then, 
$$
\max\left(-c,~  \min_{\lambda\ge 0}R(s, a, \lambda)\right) = -c 
$$
Then
\begin{equation}
\label{v}
\begin{aligned}
Q^\star(s,a) + v^\star &= \max\left(-c,~  \min_{\lambda\ge 0}R(s, a, \lambda)\right) +  \sum_{s_+\in S} P(s_+\mid s, a)h^\star(s_+) \\
&= -c +  \sum_{s_+\in S} P(s_+\mid s, a)h^\star(s_+) \\
&\le -c + c\\
&= 0
\end{aligned}
\end{equation}
On the other hand, if $a$ is such that $r^j(s,a)\ge 0$ for all $j\in [J]$, then we have that
$$
\max\left(-c,~  \min_{\lambda\ge 0}R(s, a, \lambda)\right) = \min_{\lambda\ge 0}R(s, a, \lambda)
$$
and
$$
Q^\star(s,a) + v^\star = \min_{\lambda\ge 0}R(s, a, \lambda) +  \sum_{s_+\in S} P(s_+\mid s, a)h^\star(s_+) 
$$
which is identical to (\ref{Q2}). Thus, by choosing a policy that satisfies the constraints in (\ref{optproblem2}), $r^j(s,a)\ge 0$, $Q^\star(s,a) + v^\star$ will be positive. We conclude that if the policy $\pi^\star$
implies that $Q^\star(s,a)+v^\star>0$, then it's optimal for (\ref{dual2}). Otherwise at least one of the constraints $r^j(s,a)\ge 0$ is not satisfied and the value of 
$$\max_{\pi\in \Pi}\E(Q^\star(s,\pi(s)))+v^\star=\E(Q^\star(s,\pi^\star(s)))+v^\star$$ is non positive which implies that the value of (\ref{dual2}) goes to $-\infty$ , and so (\ref{optproblem2}) is not feasible.

\end{document}